\documentclass[11pt, reqno]{amsart}
\usepackage{amscd}
\usepackage{amsmath}
\usepackage{amssymb}
\usepackage{latexsym}
\usepackage{amsthm}
\usepackage{graphicx}
\usepackage[all]{xy}       

    \SelectTips{cm}{10}     

    \everyxy={<2.5em,0em>:} 

    \xyoption{web}          

\usepackage{hyperref}

\newcommand{\arxiv}[1]{\href{http://arxiv.org/abs/#1}{\tt arXiv:\nolinkurl{#1}}}
\newcommand{\arXiv}[1]{\href{http://arxiv.org/abs/#1}{\tt arXiv:\nolinkurl{#1}}}

\newtheorem{theorem}{Theorem}[section]
\newtheorem{lemma}[theorem]{Lemma}

\newtheorem{proposition}[theorem]{Proposition}
\newtheorem{corollary}[theorem]{Corollary}

\theoremstyle{remark}
\newtheorem{remark}[theorem]{Remark}

\numberwithin{equation}{section}

\def\ii{{\bf i}}
\def\j{{\bf j}}

\def\N{\mathbb{N}}
\def\Q{\mathbb{Q}}

\def\Z{\mathbb{Z}}

\def\A{\mathcal{A}}

\def\uj{\bf j}

\def\oo{\mathcal{O}}

\def\n{\mathfrak {n}}

\def\a{\alpha}
\def\b{\beta}

\def\la{\lambda}

\def\ga{\gamma}

\def\f{\mathbf{f}}

\newcommand{\isomto}{\overset{\sim}{\rightarrow}}

\def\Hom{\,\mbox{Hom}}
\def\dimq{\dim_q}
\def\Ind{\operatorname{Ind}}
\def\Coind{\operatorname{CoInd}}
\def\seq{\,\mbox{Seq}\,}
\def\ch{\,\mbox{ch}}

\def\Seq{\,\mbox{Seq}\,}

\def\Res{\operatorname{Res}}

\def\Ext{\operatorname{Ext}}

\def\mods{\mbox{-mod}}
\def\fmod{\mbox{-fmod}}
\def\prmod{\mbox{-pmod}}

\newcommand{\map}[2]{\,{:}\,#1\!\longrightarrow\!#2}

\def\kpf{\operatorname{kpf}} 

\def\ra{\rightarrow}

\numberwithin{equation}{section}

\title[KLR algebras]{Finite Dimensional Representations of Khovanov-Lauda-Rouquier algebras I: Finite Type}
\address{}\email{petermc@math.stanford.edu}
\author{Peter J McNamara}
\date{\today}

\begin{document}

\begin{abstract}We classify simple representations of Khovanov-Lauda-Rouquier algebras in finite type. The classification is in terms of a standard family of representations that is shown to yield the dual PBW basis in the Grothendieck group. Finally, we describe the global dimension of these algebras.
\end{abstract}
\maketitle

\section{Introduction} 

Let $\mathfrak{g}$ be a complex semisimple Lie algebra. Recently, some categorifications of the upper-triangular part of the corresponding quantum group have appeared in terms of the module categories of certain families of algebras, now known as Khovanov-Lauda-Rouquier algebras (which will henceforth be referred to as KLR algebras). These algebras were introduced independently by Khovanov and Lauda \cite{khovanovlauda,klII} and Rouquier \cite{rouquier}. In this paper we restrict our attention to KLR algebras arising from Cartan data of finite type.

It is then a natural question to ask for a classification of all simple representations of these algebras. 
There are various answers to this question dating back to \cite{khovanovlauda}. Let us focus our attention on the classification in terms of Lyndon words due to
Kleshchev and Ram \cite{kleshchevram} and Hill, Melvin and Mondragon \cite{hmm}, where a family of standard modules are constructed from which the irreducibles appear as their heads. In this story, a choice is made which yields a convex ordering on the set of positive roots and it is a natural question to ask if a similar result holds for an arbitrary convex order.

It is answering this question which is the primary focus of this paper. Our main theorem is Theorem \ref{main}. For each choice of convex ordering, we produce a family of KLR-modules which categorify the dual PBW basis, and for which the simple representations appear as their heads. As pointed out by Kato \cite[Theorem 4.16]{kato}, in symmetric type this implies that the canonical basis has a positive expression in terms of any PBW basis, answering a question of Lusztig. In non-symmetric type, we instead get the result that the basis arising from this KLR categorification has a positive expression in terms of any PBW basis.

Our description of the cuspidal representations necessary to kickstart this process is perhaps best described as non-constructive, though an analysis of the proofs show that it is also
possible to obtain them via an inductive process. Unfortunately we do not have reflection functors categorifying Lusztig's automorphisms $T_i$ for representations of KLR algebras outside of types ADE in characteristic zero, where they were recently introduced by Kato \cite{kato} using geometric techniques. Possible future access to reflection functors should provide a more direct approach towards constructing the cuspidal representations.


The various classifications that arise for each choice of convex ordering are related to each other via the combinatorics of Mirkovic-Vilonen polytopes. We will not discuss this connection in this paper - this theory is developed in the paper of Tingley and Webster \cite{tingleywebster}.


Our other main result is a computation of the global dimension of a KLR algebra. In particular we prove Theorem \ref{finitecd} which states that the global dimension of the KLR algebra $R(\nu)$ is equal to the height of $\nu$. This generalises a result of Kato \cite[Theorem A]{kato}, who proves finiteness of global dimension for finite type simply-laced KLR algebras. 

We would like to acknowledge beneficial conversations with J. Brundan, D. Bump, J. Hartwig, A. Licata, T. Nevins, A. Pang, A. Ram and P. Tingley.

\section{Preliminaries}

\subsection{The Root System}

For our purposes, a Cartan datum shall consist of a finite set $I$ and a symmetric function $I\times I\to \Z$, $(i,j)\mapsto i\cdot j$ such that $i\cdot i$ is a positive even integer and $2\frac{i\cdot j}{i\cdot i}$ is a non-negative integer for any pair of distinct elements $i,j\in I$. We will extend $\cdot$ by linearity to allow for arbitrary $\Z$-linear formal combinations of elements of $I$ as arguments.

In this paper, we shall be concerned only with the case where the Cartan datum $(I,\cdot)$ is of finite type. By definition this means that the symmetric matrix $(i\cdot j)_{ij}$ is positive definite.

Associated to $(I,\cdot)$ is a root system in which the root lattice is equal to $\Z I$, the simple roots are of the form $i$ (and we will denote these $\a_i$), and the positive roots $\Phi^+$ are the roots contained in $\N I$. 

Let $W$ be the Weyl group, generated by simple reflections $\{s_i\}_{i\in I}$. Since we are assuming our Cartan datum is of finite type, $W$ is finite. Let $w_0$ be the longest element in $W$ and let $w_0=s_{i_1}s_{i_2}\cdots s_{i_N}$ be a reduced decomposition of $w_0$.

By \cite[Ch VI, \S 6, Cor 2]{bourbaki}, defining $\a_j=s_{i_N}\ldots s_{i_{j+1}} \a_{i_j}$ for each $j=1,2,\ldots ,N$ is an enumeration of the positive roots. We define a total order on $\Phi^+$ by $\a_1<\cdots < \a_N$. 

This ordering is \emph{convex} 
in the sense that the following lemma holds.

\begin{lemma}\label{convex}
 Let $\b$ be a positive root such that $\b=\sum_{i=k}^l c_i \a_i$ with each $c_i$ a nonnegative real number. Then $\a_k\leq\b\leq\a_l$ in the total ordering on $\Phi^+$.
\end{lemma}

\begin{proof}
 Let $w=s_{i_k}\cdots s_{i_N}$. By \cite[Ch VI, \S 6, Cor 2]{bourbaki} the set of positive roots $\a$ for which $w\a$ is negative is equal to the set of positive roots that are greater than or equal to $\a_k$. As $\b$ is a nonnegative linear combination of such roots, $w\b$ must also be negative, implying $\b\geq \a_k$. The other inequality is proved similarly.
\end{proof}

It is possible to show that any total ordering on $\Phi^+$ which satisfies this convexity property arises from the above construction, though we will not need this fact.


%
%
%
Throughout this paper, we will work with a fixed choice of convex ordering on $\Phi^+$, which our theorems and constructions will implicitly depend on.

Let $\a$ be a positive root that is not simple. Define a \emph{minimal pair} for $\a$ to be a pair $(\b,\ga)$ of positive roots such that $\a=\b+\ga$, $\b<\a<\ga$ and there is no pair of positive roots $\b'$, $\ga'$ with $\a=\b'+\ga'$ and $\b<\b'<\a<\ga'<\ga$.

For $\nu\in\N I$, let $\kpf(\nu)$ be the number of ways to write $\nu$ as a sum of positive roots. This is the Kostant partition function, known to equal the dimension of the $\nu$-weight space of both the enveloping algebra $U(\n)$ and its quantum analogue, where $\n$ is the nilpotent radical of a Borel subalgebra of $\mathfrak{g}$.

\begin{lemma}\label{roots}
 Let $\a$, $\delta_1,\ldots,\delta_n$ be positive roots with $n>1$ and $\a=\sum_{i=1}^n \delta_i$. Then there exists a nonempty proper subset $S\subset [n]$ such that $\sum_{s\in S}\delta_s$ and $\sum_{s\notin S} \delta_s$ are both roots.
\end{lemma}

\begin{proof}
 We proceed by induction on $n$. Thus we may assume without loss of generality that $\delta_i+\delta_j$ is not a root whenever $\delta_i\neq \delta_j$. Thus $\langle \delta_i^\vee,\delta_j\rangle \geq 0$ for all $i,j$.

Therefore $\langle\delta_i^\vee,\a\rangle=2+\sum_{j\neq i}\langle\delta_i^\vee,\delta_j\rangle\geq 2$. Hence $\a-\delta_i$ is a root and we may take $S=\{i\}$.
\end{proof}

\subsection{The Quantum Group}

For $i\in I$, let $q_i=q^{i\cdot i/2}$. The quantum integer is defined by $[n]_i= \frac{q_i^n-q_i^{-n}}{q_i-q_i^{-1}} $ and the quantum factorial is $[n]_i!=[n]_i[n-1]_i\cdots [1]_i$. For a root $\a$, we define $q_{\a}=q^{\a\cdot \a/2}$ and similarly define $[n]_\a !$.

We will work with Lusztig's twisted bialgebra $\f$. This is the associative algebra over the field $\Q(q)$ generated by elements $E_i$ for $i\in I$, subject to the quantum Serre relations
\[
\sum_{a+b=1-a_{ij}} (-1)^a E_i^{(a)}E_j E_i^{(b)} = 0
\] 
where $a_{ij}=2i\cdot j/i\cdot i$ and $E_i^{(n)}=E_i^n/[n]_i!$ is the quantum divided power.

This algebra is graded by $\N I$ where $E_i$ is homogeneous of degree $i$. Write $\f=\oplus_{\nu\in \N I} \f_\nu$. If $\nu =\sum_{i\in I}\nu_i\cdot i\in \N I$, let the height of $\nu$ be $|\nu|=\sum_{i\in I}\nu_i$.

The tensor product $\f \otimes \f$ is equipped with an algebra structure via the rule
\[
 (x_1\otimes x_2)(y_2\otimes y_2)=q^{-\mu\cdot\nu} x_1y_1\otimes x_2y_2
\] for $x_2\in \f_\mu$ and $y_1\in \f_\nu$. Define $r\map{\f}{\f\otimes\f}$ to be the unique algebra homomorphism with $r(E_i)=E_i\otimes 1+1\otimes E_i$. With this, $\f$ becomes a twisted bialgebra.

There is a nondegenerate symmetric bilinear form $(\cdot,\cdot)$ on $\f$ such that
\begin{eqnarray*}
(E_i,E_i)&=&\frac{1}{1-q_i^2} \\
(xy,z)&=&(x\otimes y,r(z)) \\
(x,yz)&=&(r(x),y\otimes z)
\end{eqnarray*}
where the bilinear form $(\cdot,\cdot)$ on $\f\otimes \f$ is given by $(x\otimes x',y\otimes y')=(x,y)(x',y')$.

Define $\f^*$ to be the graded dual of $\f$, i.e. $\f^*=\oplus_{\nu\in \N I} \f_\nu^*$.

Let $\A=\Z[q,q^{-1}]$. The twisted bialgebra $\f$ has an integral form over $\A$, denoted $\f_\A$. By definition $\f_\A$ is the $\A$-subalgebra of $\f$ generated by all divided powers $E_i^{(n)}$ for $i\in I$ and $n\in \N$. Its dual $\f^*$ inherits an integral form over $\A$, denoted $\f_\A^*$. Although $\f$ and $\f^*$ are isomorphic as twisted bialgebras over $\Q(q)$, this isomorphism does not extend to an isomorphism of their integral forms.

Consider a reduced decomposition $w_0=s_{i_1}s_{i_2}\cdots s_{i_N}$. Associated to this decomposition is a PBW basis of $\f$. This is customarily defined in terms of some algebra automorphisms $T_i$ of $U_q(\mathfrak{g})$ whose exact form shall not concern us (they are the automorphisms $T''_{i,+1}$ of \cite{lusztigbook}). We shall be content with summarising their relevant properties.

For each root $\a_k=s_{i_N}\cdots s_{i_{k+1}}\a_{i_k}$, we define a root vector $E_{\a_k}=T_{i_N}\cdots T_{i_{k+1}}E_{i_k}$. We use the notation $E_{\a}^{(n)}=E_{\a}^n/[n]_\a !$ for its divided power. Although the automorphisms $T_i$ do not preserve $\f$, each of these root vectors is an element of $\f$. The main theorem regarding PBW bases is the following:

\begin{theorem}\label{pbwbasis}
 The elements of the form $E_{\a_1}^{(m_1)}E_{\a_2}^{(m_2)}\cdots E_{\a_N}^{(m_N)}$ form an $\A$ basis of $\f_\A$, orthogonal with respect to the bilinear form $(\cdot,\cdot)$. Furthermore $(E_\a,E_\a)=(1-q_\a^2)^{-1}$.
\end{theorem}

We are more interested in the dual PBW basis of $\f_\A^*$. Use the nondegenerate pairing $(\cdot,\cdot)$ to identify $\f^*$ with $\f$ and let $E_\a^*=(1-q_\a^2)E_\a$. Then the above theorem implies that the monomials $(E^*_{\a_1})^{m_1}(E^*_{\a_2})^{m_2}\cdots (E^*_{\a_N})^{m_N}$ form an $\A$ basis of $\f_\A^*$.

\subsection{KLR Algebras}

For each $\nu\in \N I$, define
\[
\seq(\nu)=\{ (i_1,i_2,\ldots,i_k)\mid i_j\in I, \sum_{j=1}^k i_j=\nu \}
\]
For $1\leq j\leq |\nu|$, let $s_j$ be the simple reflection $(j,j+1)$ in the symmetric group on $|\nu|$ letters. This acts both on the set $\{1,2,\ldots |\nu|\}$ and on $\seq(\nu)$ in the usual way.

Fix a total ordering on $I$. We define polynomials $Q_{ij}(u,v)$ for all $i,j\in I$ by $Q_{ii}(u,v)=0$ and if $i\neq j$, $Q_{ij}(u,v)=u^{ -a_{ij}}-v^{-a_{ji}}$ if $i<j$ and $Q_{ji}(u,v)=Q_{ij}(v,u)$. Here $a_{ij}=2i\cdot j/i\cdot i$ is the Cartan integer.

Let $k$ be a field. The KLR algebra $R(\nu)$ is defined to be the associative $k$-algebra generated by elements $e_{\bf i}$, $y_j$, $\phi_k$ with ${\bf i}\in \seq(\nu)$, $1\leq j\leq |\nu|$ and $1\leq k< |\nu|$ subject to the relations
\begin{equation} \label{eq:KLR}
\begin{aligned}
& e_\ii e_{\uj} = \delta_{\ii, \uj} e_\ii, \ \
\sum_{\ii \in \Seq(\nu)}  e_\ii = 1, \\
& y_{k} y_{l} = y_{l} y_{k}, \ \ y_{k} e_\ii = e_\ii y_{k}, \\
& \phi_{l} e_\ii = e_{s_{l}\ii} \phi_{l}, \ \ \phi_{k} \phi_{l} =
\phi_{l} \phi_{k} \ \ \text{if} \ |k-l|>1, \\
& \phi_{k}^2 e_\ii = Q_{\ii_{k}, \ii_{k+1}} (y_{k}, y_{k+1})e_\ii, \\
& (\phi_{k} y_{l} - y_{s_k(l)} \phi_{k}) e_\ii = \begin{cases}
-e_\ii \ \ & \text{if} \ l=k, \ii_{k} = \ii_{k+1}, \\
e_\ii \ \ & \text{if} \ l=k+1, \ii_{k}=\ii_{k+1}, \\
0 \ \ & \text{otherwise},
\end{cases} \\[.5ex]
& (\phi_{k+1} \phi_{k} \phi_{k+1}-\phi_{k} \phi_{k+1} \phi_{k}) e_\ii\\
& =\begin{cases} \dfrac{Q_{\ii_{k}, \ii_{k+1}}(y_{k},
y_{k+1}) - Q_{\ii_{k}, \ii_{k+1}}(y_{k+2}, y_{k+1})}
{y_{k} - y_{k+2}}e_\ii\ \ & \text{if} \
\ii_{k} = \ii_{k+2}, \\
0 \ \ & \text{otherwise}.
\end{cases}
\end{aligned}
\end{equation}

An alternative interpretation of these algebras in terms of a diagrammatic calculus is given in \cite{khovanovlauda, klII}.
Since we are restricting ourselves to working in finite type, we have not given the most general form of these algebras, as developed for example in \cite{rouquier}.
The discussion in \cite{klII} shows that the choice of total ordering on $I$ is irrelevant (and indeed, it is included only to match up with the geometric picture in the simply-laced case presented in \cite{vv} and \cite{rouquier2}). 

At times we will provide references to results in \cite{khovanovlauda}. The reader should not be concerned that \cite{khovanovlauda} only works with simply-laced Cartan data as the proofs carry over to the general case. This is discussed in \cite{klII}.

The algebras $R(\nu)$ are $\Z$-graded where $\deg(e_{\bf i})=0$, $\deg(y_j e_{\bf i})=\ii_j\cdot\ii_j$ and $\deg(\phi_k e_\ii)=-\ii_k\cdot\ii_{k+1}$.

Write $R(\nu)=\oplus_{j\in \Z}R(\nu)_j$ where $R(\nu)_j$ is the $j$-th graded piece. Then each $R(\nu)_j$ is finite dimensional, and $R(\nu)$ is almost positively graded in the sense that there exists $d\in \Z$ such that $R(\nu)_j=0$ for $j<d$.

All representations of $R(\nu)$ which we will consider will be $\Z$-graded representations.
If $M$ is such a representation, let $M_i$ denote its $i$-th graded piece and $\dimq(M)=\sum_i\dim(M_i)q^i$ be its graded dimension. For $a\in \Z$, let $M\{a\}$ be the module $M$ with grading shifted by $a$, so that $M\{a\}_i=M_{a+i}$. If it is not important to us, then we do not bother keeping track of the grading shifts. Thus many maps which appear in this paper are not necessarily of degree zero. In this vein, $\Ext^i(M,N)$ will denote the Ext group in the category of ungraded modules. It inherits a grading from gradings on $M$ and $N$.

The major reason for studying KLR algebras is the existence of isomorphisms due to Khovanov and Lauda \cite{khovanovlauda} identifying their Grothendieck groups of graded finitely generated projective and finite dimensional modules,
\begin{equation}\label{pmod}
 \bigoplus_{\nu\in \N I} K_0(R(\nu)\prmod) \cong \f_\A,
\end{equation}
and
\begin{equation}\label{fmod}
 \bigoplus_{\nu\in \N I} K_0(R(\nu)\fmod) \cong \f^*_\A.
\end{equation}
The $\A$-module structure on these Grothendieck groups arises from having $q$ act by the grading shift $M\mapsto M\{1\}$.

In particular, note that the number of irreducible representations of $R(\nu)$ is equal to $\dim(\f_\nu)=\kpf(\nu)$.

Given a representation $M$ of $R(\nu)$ and $\uj\in\Seq(\nu)$, define the $\uj$-weight space of $M$ to be $e_{\uj}M$. The character of a representation $M$ is defined to be the formal sum
\[
 \ch(M)=\sum_{\uj\in\Seq(\nu)} \dimq (e_{\uj} M)[\j].
\]
We consider the character as an element of the quantum shuffle algebra, as in \cite[\S 4]{kleshchevram}. Then
the character of a finite dimensional module is equal to the image of $[M]$ under the isomorphism (\ref{fmod}) composed with the usual inclusion of $\f_\A^*$ into the quantum shuffle algebra.

There is a bar involution on $\f_\A^*$ which is easiest to describe as the restriction of the involution $\sum_\j f_\j(q)[\j]\mapsto \sum_\j f_\j(q^{-1})[\j]$ on the quantum shuffle algebra.

\subsection{Induction and Restriction}

Let $\la$ and $\mu$ be two elements of $\N I$. Then there is a non-unital inclusion of algebras $R(\la)\otimes R(\mu)\to R(\la+\mu)$. In the diagrammatic picture, this is defined by placing diagrams next to each other. 
Let $e_{\la\mu}$ denote the image of the identity of $R(\la)\otimes R(\mu)$ under this inclusion.

The restriction functor $\Res_{\la\mu}\colon R(\la+\mu)\mods\to R(\la)\otimes R(\mu)\mods$ is defined by $M\mapsto e_{\la\mu}M$ on objects and the obvious map on morphisms.

The induction functor $\Ind_{\la\mu}\colon R(\la)\otimes R(\mu)\mods\to R(\la+\mu)\mods$ is defined by $$\Ind_{\la\mu}(M\boxtimes N)= R(\la+\mu)e_{\la\mu} \bigotimes_{R(\la)\otimes R(\mu)} (M\boxtimes N).$$
We often write $M\otimes N$ for $\Ind(M\boxtimes N)$ and $M^{\otimes n}$ for $M\otimes M\otimes \cdots \otimes M$ ($n$ times).

Via the isomorphisms (\ref{pmod}) and (\ref{fmod}) the operations of induction and restriction induce multiplication and $r$ respectively. For a precise statement of these results, one may wish to consult \cite[\S 4]{kleshchevram}.

The coinduction functor $\Coind_{\la\mu}\map{R(\la)\otimes R(\mu)\mods}{R(\la+\mu)\mods}$ is defined by $$\Coind_{\la\mu}(M\boxtimes N)= \Ind_{\mu\la}(N\boxtimes M)\{(\la\cdot\mu)\}.$$

More generally, given $\la_1,\la_2,\ldots,\la_k\in \N I$, there are induction and restriction functors
\[
 \Ind_{\la_1,\ldots,\la_k}\map{\bigotimes_{i=1}^k R(\la_i)\mods }{R\left(\sum_{i=1}^k \la_i\right)\mods}
\]
and
\[
 \Res_{\la_1,\ldots,\la_k}:R\left(\sum_{i=1}^k \la_i\right)\mods \to \bigotimes_{i=1}^k R(\la_i)\mods.
\]
These functors satisfy the usual associativity conditions.

These functors have the adjunction properties that one expects from their names.

\begin{theorem}\cite[Theorem 2.2]{laudavazirani}
 The induction functor $\Ind$ is left adjoint to the restriction functor $\Res$, while the coinduction functor $\Coind$ is right adjoint to restriction.
\end{theorem}

\begin{proposition}\cite[Corollary 2.17]{khovanovlauda}
 The functors of induction and restriction take projective representations to projective representations.
\end{proposition}

\begin{corollary}\label{extadjunction}
 There are natural isomorphisms
\[
 \Ext^i(\Ind(A\boxtimes B),C)\cong\Ext^i(A\boxtimes B,\Res(C))
\] and 
\[
 \Ext^i(\Res(A),B\boxtimes C)\cong\Ext^i(A,\Ind(C\boxtimes B)).
\]

\end{corollary}

The following Mackey-style result will play an important role. The proof is the same as for the special case considered in \cite{khovanovlauda}. 

\begin{proposition}\cite[Proposition 2.18]{khovanovlauda}\label{mackey}
Let $\la_1,\ldots,\la_k,\mu_1\ldots,\mu_l\in\N I$ be such that $\sum_i \la_i=\sum_j \mu_j$. Then the composite functor
$ \Res_{\mu_1,\ldots,\mu_l}\circ\Ind_{\la_1,\ldots,\la_k}
$ has a filtration indexed by tuples $\nu_{ij}$ satisfying $\la_i=\sum_j \nu_{ij}$ and $\mu_j=\sum_i\nu_{ij}$.
The subquotients of this filtration are isomorphic, up to a grading shift, to the composition $\Ind_\nu^\mu\circ \tau\circ \Res_\nu^\la$ where
$\Res_{\nu}^\la\map{\otimes_i R(\la_i)\mods}{\otimes_i(\otimes_jR(\nu_{ij}))\mods}$ is the tensor product of the $\Res_{\nu_{i\bullet}}$, $\tau\map{\otimes_i(\otimes_jR(\nu_{ij}))\mods}{\otimes_j(\otimes_iR(\nu_{ij}))\mods}$ is given by permuting the tensor factors and $\Ind_\nu^\mu\map{\otimes_j(\otimes_iR(\nu_{ij}))\mods}{\otimes_j R(\mu_j)\mods}$ is the tensor product of the $\Ind_{\nu_{\bullet i}}$.
\end{proposition}

%
%
%
%
%

\begin{proposition}\label{characterinjective}
 For any representations $A$ and $B$, the set of Jordan-Holder constituents of $A\otimes B$ is equal, up to a grading shift, to the set of Jordan-Holder constituents of $B\otimes A$.
\end{proposition}

\begin{proof}
 This is immediate from the ungraded version of \cite[Theorem 3.17]{khovanovlauda}.
\end{proof}

\section{Simple Representations}

For the duration of this section, we fix once and for all a reduced decomposition for $w_0$. By the discussion in the previous section, this choice induces a convex order on $\Phi^+$ and a dual PBW basis of $\f^*$. As a result, all results in this section will depend on the choice of this decomposition. 

Let $\a_1<\cdots < \a_N$ be the enumeration of $\Phi^+$ determined by the choice of reduced decomposition. For each positive root $\a$, let $E_\a$ denote the corresponding root vector in the PBW basis of $\f$ and let $E_\a^*$ be the corresponding vector in the dual PBW basis of $\f^*$.

Let $\N^{\Phi^+}$ be the set of functions from $\Phi^+$ to $\N$. Given $m\in \N^{\Phi^+}$, we identify $m$ with an $N$-tuple of natural numbers $(m_1,\ldots,m_N)$ via $m_i=m(\a_i)$. We will use the notation $1_\a$, or sometimes $\a$ to denote the characteristic function of $\a$. The support of $m\in \N^{\Phi^+}$ is defined to be the set of positive roots $\a$ for which $m(\a)\neq 0$. For $m\in \N^{\Phi^+}$, we let $|m|=\sum_\a m(\a) \a$.

We put a lexicographic ordering on $\N^{\Phi^+}$, namely $(m_1,\ldots,m_N)< (n_1,\ldots,n_N)$ if $n_l>m_l$ where $l$ is the largest index $i$ for which $m_i\neq n_i$.

There is also the opposite lexicographic ordering on $\N^{\Phi^+}$, given by $(m_1,\ldots,m_N)<' (n_1,\ldots,n_N)$ if $n_l>m_l$ where $l$ is the smallest index $i$ for which $m_i\neq n_i$.

For $m\in \N^{\Phi^+}$ define $\Res_{m}=\Res_{m_1 \a_1,m_2\a_2,\ldots,m_N\a_n}$.

We are now in a position to state our main theorem. It generalises \cite[Theorem 7.2, Proposition 7.4]{kleshchevram} and \cite[Proposition 4.2.1]{hmm}. We will prove it, together with all other lemmas in this section simultaneously by induction on the height of $\nu$.

\begin{theorem}\label{main}
 For each positive root $\a$, there exists a simple representation $S_\a$ of
$R(\a)$ such that
\begin{enumerate}
\item The image of the class of $S_\a$ under the isomorphism (\ref{fmod}) is the dual PBW basis element $E_\a^*$.
 \item For each $m\in \N^{\Phi^+}$, the representation $\Delta(m)=S_1^{\otimes m_1}\otimes
S_2^{\otimes m_2}\otimes\cdots \otimes S_N^{\otimes m_N}$ has a unique
irreducible quotient $L(m)$.
\item The simple representations $L(m)$ thus constructed form a set of
representatives of isomorphism classes of simple representations of
KLR algebras.
\item The representation $\nabla(m)=S_N^{\otimes m_N}\otimes
S_{N-1}^{\otimes m_{N-1}}\otimes\cdots \otimes S_1^{\otimes m_1}$ has $L(m)$ as its socle.
\item Any simple constituent $L(m')$ of a composition series of $\Delta(m)$ or $\nabla(m)$ satisfies $m\leq m'$ and $m\leq' m'$.
Furthermore $L(m)$ appears in such a composition series with multiplicity one.

\end{enumerate}
\end{theorem}

The simple representations $S_\a$ of $R(\a)$ appearing in this theorem will be referred to as cuspidal representations. We begin by discussing how to construct these cuspidal representations.

We know that the number of simple representations of $R(\a)$ is equal to $\kpf(\a)$. There are $\kpf(\a)-1$ representations $\Delta(m)$ of $R(\a)$ which can be assumed to already have been constructed by inductive hypothesis.
By the above theorem, each $\Delta(m)$ has irreducible head $L(m)$, and these simple representations are pairwise non-isomorphic. The invocation of our main theorem at this stage will turn out to be valid, as is determinable by an inspection of its proof. The cuspidal representation $S_\a$ is then defined to be the unique simple representation of $R(\a)$ that has not already been constructed.

This process determines $S_\a$ up to a grading shift. We normalise the grading shift as in \cite{khovanovlauda}, which ensures that $[S_\a]$ is bar-invariant.

There is an alternative construction of these cuspidal representations using reflection functors in the simply-laced characteristic zero case due to Kato \cite{kato}.

The following lemma underlies the entire argument.

\begin{lemma}\label{ressa}
 Let $\a$ be a positive root. If $\b,\ga\in\N I$ are such that $\Res_{\beta,\ga}S_\a \neq 0$, then $\beta$ is a
sum of roots greater than or equal to $\a$ and $\ga$ is a sum of roots less than or equal to $\a$.
\end{lemma}
\begin{proof}
We will prove the statement about $\ga$ being a sum of roots less than or equal to $\a$, and the corresponding statement for $\b$ will follow similarly. We proceed by induction on the height of $\a$, assuming by induction that Theorem \ref{main} is known for all $\nu$ with $|\nu|<|\a|$.

Let $L(m)\boxtimes L(m')$ be a simple subrepresentation of $\Res_{\b,\ga}(S_\a)$.
Let $\ga'$ be the largest root in the support of $m'$. Then $\Res_{\a-\ga',\ga'}(S_\a)\neq 0$. To prove that $\ga$ is a sum of roots less than or equal to $\a$, it suffices to prove that $\ga'\leq \a$.
We may thus replace $(\b,\ga)$ by $(\a-\ga',\ga')$.
In this manner we may assume without loss of generality that $\ga$ is a root. Furthermore, without loss of generality assume that $\ga$ is the largest possible root for which $\Res_{\a-\ga,\ga}S_\a \neq 0$.
This discussion then shows that $L(m')=S_\ga$.

Let $\delta$ be the largest root appearing in the support of $m$. 
Since $L(m)$ is a quotient of $\Delta(m)$, by adjunction we can find a nonzero map $M\boxtimes S_\delta\to \Res_{\b-\delta,\delta}(L(m))$ for some $R(\b-\delta)$-module $M$.

This induces a nonzero map $M\boxtimes S_\delta\boxtimes S_\ga\to \Res_{\b-\delta,\delta,\ga}(S_\a)$ which by adjunction induces a nonzero map $M\boxtimes \Ind(S_\delta\boxtimes S_\gamma)\to \Res_{\b-\delta,\delta+\ga}(S_\a)$.

If $\delta=\b$, then this realises $S_\a$ as a quotient of $\Ind(S_\b\boxtimes S_\ga)$. By construction of $S_\a$, we must have that $\ga\leq\b$. Since $\a=\b+\ga$, we obtain from Lemma \ref{convex} that $\ga\leq\a$ as required.

Now assume that we are not in the $\delta=\b$ case. Let $L(m'')$ be a simple module in the composition series of $\Ind(S_\delta\boxtimes S_\ga)$ that is not in the kernel of the map $M\boxtimes \Ind(S_\delta\boxtimes S_\gamma)\to \Res_{\b-\delta,\delta+\ga}(S_\a)$. By assumption on the maximality of $\ga$, every $\ga''$ in the support of $m''$ satisfies $\ga''\leq \ga$. Thus $\delta+\ga$ is a sum of roots less than or equal to $\ga$, so using Lemma \ref{convex}, $\delta\leq \ga$.

Since $\delta\leq \ga$, the module $\Delta(m+1_\ga)$ surjects onto $S_\a$, contradicting the construction of $S_\a$ and thus the lemma is proved.
\end{proof}

\begin{lemma}\label{restriction}
For $m,n\in \N^{\Phi^+}$, we have
\[ 
\Res_{n}\Delta(m) = \begin{cases}
                                               0 & \mbox{if } n > m \mbox{ or } n>'m\\
S_1^{\otimes m_1}\boxtimes\cdots\boxtimes S_N^{\otimes m_N} & \mbox {if } n=m.
                                              \end{cases}
\]
\end{lemma}
\begin{proof}

Suppose $n\geq m$. Let $l$ be the largest index with $n_l\neq 0$. Then $m_i=0$ for all $i>l$.
Suppose that $\Res_n\Delta(m)\neq 0$. Proposition \ref{mackey} identifies the subquotients in a filtration for $\Res_n\Delta(m)$, one of which must be nonzero. This nonzero subquotient is indexed by a set of elements $\nu_i\in \N I$ with $n_l\a_l = \sum_i \nu_i$ and for each $i$, there is some $j\leq l$ with $\Res_{\a_j-\nu_i,\nu_i}S_{\a_j}\neq 0$. By Lemma \ref{ressa}, each $\nu_i$ is a sum of roots less than or equal to $\a_j$, which is in turn less than or equal to $\a_l$.

We have expressed $n_l\a_l$ as a sum of roots less than or equal to $\a_l$. Since the ordering on positive roots is convex, this can happen in exactly one way. Thus we must have $m_l\geq n_l$.

For $\Res_n\Delta(m)$ to be nonzero, the only option now is that $m_l=n_l$. Let $m'=m-m_l1_{\a_l}$ and $n'=n-n_l1_{\a_l}$. In this case, the above argument shows that we necessarily have
\[
 \Res_n\Delta(m) = \Res_{n'}\Delta(m')\boxtimes S_{\a_l}^{\otimes m_l}.
\]

By induction on $l$, we have proved this Lemma for the ordering $>$ on $\N^{\Phi^+}$. Similarly, we obtain the result for the ordering $>'$.
\end{proof}

The following lemma generalises \cite[Lemma 6.6]{kleshchevram}.

\begin{lemma}\label{cuspidalpower}
 The representation $S_\a^{\otimes n}$ of $R(n\a)$ is irreducible.
\end{lemma}

\begin{proof}

 By Lemma \ref{restriction}, $\Res_m(S_\a^{\otimes n})=0$ unless the support of $m$ is $\a$. Since $\Res_m(L(m'))=0$ unless the support of $m'$ is $\a$, only one isomorphism class of simple modules can appear in the composition series of $S_\a^{\otimes n}$. So to prove irreducibility, it suffices to prove that $[S_\a^{\otimes n}]$ is indivisible in the Grothendieck group.

If $n=1$, then this lemma is true by definition. Otherwise, by induction applied to Theorem \ref{main}(1), we may assume $[S_\a]= E_\a^*$.
Thus $[S_\a^{\otimes n}]=(E_\a^*)^n$ which is indivisible in $\f_\A^*$ as required.
\end{proof}

We now prove Theorem \ref{main}

\begin{proof}[(2):]
Let $Q$ be a nonzero quotient of $\Delta(m)$. By adjunction there is a nonzero morphism
$S_1^{\otimes m_1}\boxtimes S_2^{\otimes m_2}\boxtimes\cdots \boxtimes
S_N^{\otimes m_n}\rightarrow \Res_m(Q)$. Lemma \ref{cuspidalpower} implies that the source of this morphism is irreducible, hence this map is injective.

If $\Delta(m)$ had a reducible head, then there would be a surjection $\Delta(m)\to Q\oplus Q'$ with $Q$ and $Q'$ nonzero. By exactness of the restriction  functor there is a surjection $\Res_m\Delta(m)\to\Res_mQ\oplus \Res_mQ'$.

Lemma \ref{restriction} implies that $\Res_m(\Delta(m))=S_1^{\otimes m_1}\boxtimes\cdots\boxtimes S_N^{\otimes m_N}$ which is simple, while the preceding discussion tells us that both $\Res_m Q$ and $\Res_m Q'$ contain $S_1^{\otimes m_1}\boxtimes\cdots\boxtimes S_N^{\otimes m_N}$ as a submodule. This contradicts the existence of such a surjection, proving (2).
\end{proof}

\begin{proof}[(3):]
The proof of (2) above also shows that $\Res_m L(m)=S_1^{\otimes m_1}\boxtimes\cdots\boxtimes S_N^{\otimes m_N}$. Lemma \ref{restriction} and the exactness of the restriction functor tell us that $\Res_{m'}L(m)=0$ unless $m\geq m'$.

These facts imply that the set of representations $\{ L(m)\mid |m|=\nu\}$ are pairwise nonisomorphic. There are $\kpf(\nu)$ simple representations of $R(\nu)$ in this set and by (\ref{fmod}), the algebra $R(\nu)$ has exactly $\kpf(\nu)$ simple representations.
\end{proof}

\begin{proof}[(4):]
 Since $\Res_m L(m) = S_1^{\otimes m_1}\boxtimes\cdots\boxtimes S_N^{\otimes m_N}$, by the coinduction adjunction we obtain a nonzero map $L(m)\to\nabla(m)$ which is injective as $L(m)$ is simple. The remainder of the proof proceeds in exactly the same fashion as the proof of (2) above.
\end{proof}

\begin{proof}[(5):]
 Suppose $L(m')$ appears in a composition series for $\Delta(m)$. We apply the exact functor $\Res_{m'}$. Since $\Res_{m'}L(m')\neq 0$ this implies $\Res_{m'}\Delta(m) \neq 0$. By Lemma \ref{restriction}, $m\geq m'$ and $m\geq' m'$. For the case when $m=m'$, we get a multiplicity of one since $\Res_m\Delta(m)=\Res_m L(m)$.

The statement for $\nabla(m)$ follows either by a similar argument or by Proposition \ref{characterinjective}.
\end{proof}

\begin{proof}[(1):]
If $\a$ is simple, $R(\a_i)\cong k[x]$ with $x$ homogeneous of degree $\a\cdot\a$ and this is an easy calculation. Now suppose $\a$ is a positive root that is not simple. Let $(\b,\ga)$ be a minimal pair for $\a$. 
We will be performing an induction on the height of the root $\a$.

Let $L(m)$ be a simple constituent of a composition series of $S_\b\otimes S_\ga$ or equivalently, by Proposition \ref{characterinjective}, of $S_\ga\otimes S_\b$. The element $m$ yields in an obvious manner a way to write $\a$ as a sum of positive roots. If $m\neq 1_\a$, by Lemma \ref{roots} we can write $\a$ as a sum of two positive roots $\b'$, $\ga'$, each of which is a sum of roots in the support of $m$.

By Theorem \ref{main}(5), the support of $m$ only contains roots between $\b$ and $\ga$ inclusive. Thus $\b\leq \b'$ and $\ga'\leq\ga$. By our assumption that $(\b,\ga)$ is a minimal pair, the only possibility is that $\b=\b'$ and $\ga=\ga'$, which can only occur if $m=1_\b+1_\ga$.

Let us write $L_{\b\ga}$ for $L(1_\b+1_\ga)$. Consider the composition
\[
 S_\b\otimes S_\ga \twoheadrightarrow L_{\b\ga} \hookrightarrow S_\ga\otimes S_\b \{\b\cdot \ga\}.
\]
Since $L_{\b\ga}$ appears only once in any composition series for $S_\b\otimes S_\ga$ or $S_\ga\otimes S_\b$, we have just shown that the kernel and cokernel of this composite morphism can only have $S_\a$ appearing in their composition series.

By inductive hypothesis, $[S_\b]=E_\b^*$ and $[S_\ga]=E_\ga^*$. Thus $E_\b^*E_\ga^*-q^{\b\cdot\ga}E_\ga^*E_\b^*$ is a multiple of $[S_\a]$.

At $q=1$, the specialisations of $E_\b$ and $E_\ga$ are nonzero vectors in the weight spaces $\mathfrak{g}_\b$ and $\mathfrak{g}_\ga$ respectively. Since $\b+\ga$ is a root, these specialisations do not commute. As $E_\b^*=(1-q_\b^2)E_\b$ and $E_\ga^*=(1-q_\ga^2)E_\ga$ this implies that $E_\b^* E_\ga^*-q^{\b\cdot\ga}E_\ga^*E_\b^*\neq 0$.

By the Levendorskii-Soibelman formula \cite[Proposition 5.5.2]{ls} and an argument similar to that we just used to prove $S_\b\otimes S_\ga$ has only two possible simple constituents, $E_\b^* E_\ga^*-q^{\b\cdot\ga}E_\ga^*E_\b^*$ is a multiple of $E_\a^*$. Thus we have shown that $[S_\a]$ is a nonzero multiple of $E_\a^*$.

Since $E_\a^*$ lies in an $\A$-basis of $\f_\A^*$, $[S_\a]\in \A E_\a^*$. Since $S_\a$ is simple, $[S_\a]$ is indivisible in $\f_\A^*$.
Thus $[S_\a]=\pm q^n E_\a^*$ for some integer $n$. Normalising the grading on $S_\a$ as in \cite{khovanovlauda} forces $[S_\a]$ to be bar-invariant. As $E_\a^*$ is also bar-invariant, this forces $n=0$. 

The rest of the proof is dedicated to removing the sign ambiguity. 
First we shall consider the case where our Cartan datum is symmetric. This allows us to make use of the geometric interpretation of KLR algebras due to \cite{vv,rouquier2}, which shows that $[S_\a]$ lies in the dual canonical basis. Since $E_\a^*$ is known to lie in the dual canonical basis, we're done in this case. 

Now we return to the general case. We seek to massage our problem into one which is amenable to the technique of folding.

For $\ii = (i_1,\ldots, i_n)\in \Seq(\a)$, let $E_\ii= E_{i_1}\cdots E_{i_n}$. Note that for any representation $M$ of $R(\a)$, we have
$(E_{\ii},[M])=\dimq(e_{\ii}M)\in \N[q,q^{-1}]$. Since we already know $[S_\a]=\pm E_\a^*$, this means that the statement $[S_\a]=E_\a^*$ is equivalent
to having $(E_\ii,E^*_\a)\in \N[q,q^{-1}]$ for all $\ii\in \Seq(\a)$.

We now specialise $\f_\A$ and $\f_\A^*$ at $q=1$. Let $G$ be a Chevalley group corresponding to our Cartan datum, $N$ the unipotent radical of a Borel subgroup and $\n$ the Lie algebra of $N$. Then $\f_\A$ specialises to $U_\Z(\mathfrak{n})$, a $\Z$-form (with divided powers) of the universal enveloping algebra of $\n$, 
and $\f_\A^*$ specialises to $\Z[N]$, the affine coordinate ring of $N$. 
The canonical pairing $(\cdot,\cdot):\f_\A\times \f_\A^* \to \A $ specialises to the pairing
\begin{equation}\label{pairing}
 (\cdot,\cdot)\map{U_\Z(\mathfrak{n})\times \Z[N]}{\Z}, \ \quad (X,f)=(Xf)(1).
\end{equation}

For each positive root $\b$, let $X_\b$ denote the specialisation of $E_\b$ to $U_\Z(\n)$. If $\b$ is simple then $X_\b$ is the usual Chevalley generator, while for general $\b$, we have $X_\b\in \n_\b$.

Since $N$ is the product of its root subgroups, we can define a function
$Z_\b\in\Z[N]$ by
\[
 Z_\b\left(\prod_{i=1}^{N} \exp ( x_{\a_i}X_{\a_i} )\right)=x_\b.
\]

For any two positive roots $\b$ and $\ga$ we compute
\[
 (X_\ga,Z_\b)=\frac{d}{dt}\left. \left(Z_\b \exp(tX_\ga)\right ) \right |_{t=0}=\delta_{\b\ga}
\]
and thus $Z_\b$ is the specialisation of the dual PBW element $E_\b^*$.

Let $e_\ii$ be the product of the Chevalley generators that is the specialisation of $E_\ii$. Then the statement $[S_\a]=E_\a^*$ is now equivalent to the statement that $(e_\ii,Z_\a)\in \N$ for all $\ii\in \Seq(\a)$. It is this last statement that is amenable to folding.

From $(I,\cdot)$, we can construct a simply laced Cartan datum $(\tilde{I},\cdot)$ together with an automorphism $\sigma$ of $(\tilde{I},\cdot)$ such that $i\cdot j=0$ whenever $i$ and $j$ are in the same $\sigma$-orbit. We now discuss the relationsip between $(\tilde{I},\cdot)$ and $({I},\cdot)$ - throughout we use a tilde to denote a construction performed using $(\tilde{I},\cdot)$.

The set $I$ is equal to the set of orbits of $\sigma$ on $\tilde I$. The group $N$ appears as the fixed point set $(\tilde{N})^\sigma$ with a similar statement for $\n$. The relationship between the Chevalley generators is, for $\oo\in I$, 
\begin{equation}\label{expand}
 e_\oo = \sum_{i\in\oo}e_i.
\end{equation}
Given a long word decomposition for $(I,\cdot)$, one may obtain a long word decomposition for $(\tilde{I},\cdot)$ by replacing all occurrences of the simple reflection $s_\oo$ by $\prod_{i\in\oo} s_i$. The function $Z_\b$ on $N$ is now obtainable as the restriction of a similarly constructed function $\tilde{Z}_\ga$ on $\tilde N$.

If $H$ is a subgroup of $G$ with corresponding Lie algebras $\mathfrak{h}$ and $\mathfrak{g}$ respectively, then writing $i$ for the inclusion of $U_\Z(\mathfrak{h}) \to U_\Z(\mathfrak{g})$ and $\pi$ for the projection $\Z[G]\to \Z[H]$, the pairing (\ref{pairing}) behaves via $$(iX,f)=(X,\pi f).$$


Let $\ii\in\Seq(\a)$. We expand $e_\ii$ using (\ref{expand}) into a positive sum of $\tilde{e}_{\uj}$'s and write $Z_\a$ as the restriction of some $\tilde Z_\a$. This reduces the computation of $(e_\ii,Z_\a)$ into a positive sum of similar expressions for the simply laced Cartan datum $(\tilde{I},\cdot)$. Since positivity is already known in the simply laced case, we obtain $(e_\ii,Z_\a)\in\N$, and hence $[S_\a]=E_\a^*$, completing our proof.
\end{proof}
\section{Global Dimension of KLR Algebras}

The main aim of this section is Theorem \ref{finitecd} which computes the global dimension of a finite type KLR algebra. We give a complete proof except in types $D$ and $E$ when the ground field $k$ is of positive characteristic. The extra computations needed to cover these final cases are included in \cite{bkm}. First we need some Lemmas.

\begin{lemma}\label{resgammabeta}
Let $\a$ be a positive root that is not simple and let $(\b,\ga)$ be a minimal pair for $\a$.
Then every simple subquotient of $\Res_{\ga\b}(S_\a)$ is isomorphic to $S_\ga \boxtimes S_\b$.
\end{lemma}

\begin{proof}
 We expand the class of $\Res_{\ga\b}(S_\a)$ in the PBW basis.
\[
 [ \Res_{\ga\b}(S_\a) ]=\sum_{\substack{ \ga_1,\ldots,\ga_l \\ \b_1,\ldots,\b_k}} c_{\ga_\bullet \b_\bullet} E_{\ga_1}\cdots E_{\ga_l} \otimes E_{\b_1}\cdots E_{\b_k}.
\] where $\b_1\leq\cdots \leq \b_k$ and $\ga_1\leq\cdots\leq\ga_l$.

By Theorem \ref{main}(1), we can extract the coefficient $c_{\ga_\bullet \b_\bullet}$ by
\[
 c_{\ga_\bullet \b_\bullet}=([ \Res_{\ga\b}(S_\a) ],E_{\ga_1}\cdots E_{\ga_l} \otimes E_{\b_1}\cdots E_{\b_k})=(E_\a^*,E_{\ga_1}\cdots E_{\ga_l}E_{\b_1}\cdots E_{\b_k}).
\]
The Levendorskii-Soibelman formula gives us an algorithm for expanding the monomial $E_{\ga_1}\cdots E_{\ga_l}E_{\b_1}\cdots E_{\b_k}$ in the PBW basis. Each monomial $E_{\delta_1}\cdots E_{\delta_m}$ which appears at any point in this computation must have $\delta_1\leq \ga_1$ and $\delta_m\geq \b_k$.

By the orthogonality of the PBW basis, to obtain a contribution to this pairing, the term $E_\a$ must appear in the above expansion. The only way that $E_\a$ can appear via repeated application of the Levendorskii-Soibelman formula is if at some intermediate point in the computation, a monomial $E_{\ga'}E_{\b'}$ with $\ga'>\b'$ appeared. Necessarily $\b'+\ga'=\a$.

By our observation on the possible monomials which can appear, we must have $\ga'\leq \ga_1\leq \ga$ and $\b'\geq \b_k\geq \b$. By the assumption that the pair $(\b,\ga)$ is minimal, this implies that these inequalities are all equalities, and hence $\ga_\bullet =\ga$ and $\b_\bullet = \b$.

Thus the class of $\Res_{\ga\b}(S_\a)$ is a multiple of $E_\ga\otimes E_\b$, which is enough to prove this Lemma.
\end{proof}
%

As in the previous section, all results that we prove will be proved simultaneously by an induction on $\nu$. 
Unlike the previous section, we often need to work with a special long word decomposition, matching that of \cite{hmm}. Since our main goal is to prove Theorem \ref{finitecd}, a statement which is independent of any long word decomposition, this shall not harm us. 

We now pause to explain the relevant contents of \cite{hmm}. 
For each positive root $\a$, let $\ii_\a$ be the corresponding good Lyndon word in $\Seq(\a)$ via the bijection of \cite[Proposition 2.3.5]{hmm}. Considering the lexicographic ordering on these good Lyndon words induces a convex ordering on the set of positive roots.
In \cite{hmm}, the set of cuspidal representations (except for the highest root in $E_8$ which doesn't concern us) are explicitly computed for this choice of ordering on $\Phi^+$, which we call the HMM-ordering.

In the next few Lemmas, we restrict ourselves to types $B$, $C$, $F$ and $G$ since it will transpire that we have a way of bypassing these Lemmas in the simply-laced case. There is no mathematical reason to make this restriction, we could include all finite type root systems if needed at the cost of checking more cases in the proofs (and the alert reader will notice that we do also provide proofs in type $A$). 

\begin{lemma}\label{ses}
Suppose our KLR algebras are of type $B$, $C$, $F$ or $G$ and that we are working with the HMM-ordering.
Let $\a$ be a positive root that is not simple and let $(\b,\ga)$ be the minimal pair for $\a$ described in the appendix. Then there are short exact sequences 
\begin{eqnarray}\label{ses1}
 &0\rightarrow L_{\b\ga} \ra S_\ga \otimes S_\b \ra S_\a \ra 0, \\
 \label{ses2} &0\ra S_\a \ra S_\b\otimes S_\ga \ra L_{\b\ga}\ra 0.
\end{eqnarray}
\end{lemma}

Technically speaking, there are grading shifts to be added to these two short exact sequences. Since we shan't need to know precisely what they are, we shall ignore them. 

\begin{remark}
 This result is in fact true for all convex ordering and all minimal pairs in all finite type KLR algebras. This is proved in \cite[Theorem 4.7]{bkm}, where the appropriate grading shifts are also included.
\end{remark}

\begin{proof}

By the argument in the proof of Theorem \ref{main}(1), we know that $S_\b\otimes S_\ga$ has a unique simple quotient $L_{\b\ga}$ and that the other simple representations in any composition series are all isomorphic to $S_\a$.

An inspection of the tables in the appendix together with the computation of the characters of each cuspidal module from \cite[\S 5]{hmm} shows that in each case, the dimension of the $\ii_\a$ weight space of $S_\a$ is equal to the dimension of the $\ii_\a$ weight space of $S_\b\otimes S_\ga$. Hence $S_\a$ can only appear once in a composition series of $S_\b\otimes S_\ga$, proving the Lemma.
\end{proof}



\begin{lemma}\label{minpairext} Suppose our KLR algebras are of type $B$, $C$, $F$ or $G$ and that we are working with the HMM-ordering.
Let $\a$ be a positive root that is not simple, and let $(\b,\ga)$ be the minimal pair for $\a$ given in the tables in the appendix. Then $$\dim\Ext^i(S_\ga\otimes S_\b,S_\a)=
\begin{cases}
1 \mbox{ if } i=0,2,\\
2 \mbox{ if } i=1,\\ 
0 \mbox{ if } i\geq 3.
\end{cases}
$$
\end{lemma}
\begin{proof}
An inspection of the tables in the appendix together with the character computations of \cite{hmm} shows that in each case, the multiplicity of $S_\ga \boxtimes S_\b$ in $\Res_{\ga\b}S_\a$ is at most three. Recall that Lemma \ref{resgammabeta} tells us that every simple subquotient of $\Res_{\ga\b}S_\a$ is isomorphic to $S_\ga \boxtimes S_\b$.

First consider the case where $\Res_{\ga\b}S_\a\cong S_\ga \boxtimes S_\b$. Then by adjunction,
\[
 \Ext^*(S_\ga\otimes S_\b,S_\a)=\Ext^*(S_\ga \boxtimes S_\b,S_\ga \boxtimes S_\b)=\Ext^*(S_\ga,S_\ga)\otimes \Ext^*(S_\b,S_\b).
\]

By an induction on the height of the root applied to Proposition \ref{cuspidalselfext}, both $\Ext^*(S_\ga,S_\ga)$ and $\Ext^*(S_\b,S_\b)$ are isomorphic to $k[x]/(x^2)$ with $x$ in homological degree one, proving the Lemma in this case.

Now consider the case where $\Res_{\ga\b}S_\a$ has a composition series of length two. Thus it lies in a short exact sequence
\begin{equation}\label{2ext}
 0\to S_\ga\boxtimes S_\b \to \Res_{\ga\b}S_\a \to S_\ga\boxtimes S_\b \to 0. 
\end{equation}

By Lemma \ref{ses}, $\Hom(S_\ga\otimes S_\b,S_\a)$ is one dimensional. By adjunction this implies $\Hom(S_\ga \boxtimes S_\b,\Res_{\ga\b}S_\a)$ is one dimensional and hence this short exact sequence does not split.

Consider the long exact sequence obtained by applying $\Hom(S_\ga\boxtimes S_\b,-)$ to the sequence (\ref{2ext}).
The boundary map $\delta\map{\Ext^1(S_\ga\boxtimes S_\b,S_\ga\boxtimes S_\b)}{\Ext^2(S_\ga\boxtimes S_\b,S_\ga\boxtimes S_\b)}$ is given by multiplication by the class of the extension (\ref{2ext}).
Again, by inductive hypothesis applied to Proposition \ref{cuspidalselfext}, $\Ext^*(S_\ga\boxtimes S_\b,S_\ga\boxtimes S_\b)\cong\Ext^*(S_\ga,S_\ga)\otimes \Ext^*(S_\b,S_\b)\cong k[x]/(x^2) \otimes k[y]/(y^2)$ with $x$ and $y$ in degree one. In this algebra, multiplication by any nonzero degree one element surjects onto the degree two piece. Since the extension (\ref{2ext}) is non-split, this implies that the boundary map is surjective.

It is now a routine procedure to extract the statement of this Lemma from the long exact sequence under consideration using the results we have just proved.

Finally, let us consider the case where $\Res_{\ga\b}S_\a$ has a composition series of length three. Then there are short exact sequences
\begin{equation}\label{xext}
 0\to X \to \Res_{\ga\b}S_\a \to S_\ga\boxtimes S_\b \to 0
\end{equation} and
\begin{equation*}\label{yext}
 0\to S_\ga\boxtimes S_\b \to \Res_{\ga\b}S_\a \to Y \to 0
\end{equation*} where $X$ and $Y$ are self-extensions of $S_\ga\boxtimes S_\b$.

By adjunction $\Hom(\Res_{\ga\b}S_\a,S_\ga\boxtimes S_\b)=\Hom(S_\a,S_\b\otimes S_\ga)$ and Lemma \ref{ses} implies that this latter space is one dimensional.
Thus $Y$ is a non-split self-extension of $S_\ga\boxtimes S_\b$. Similarly $X$ is a non-split self-extension of $S_\ga\boxtimes S_\b$.

We apply $\Hom(S_\ga\boxtimes S_\b,-)$ to the sequence (\ref{xext}) to obtain a long exact sequence of Ext groups. Since $X$ is non-split, we can use the argument from the length two case to deduce the dimensions of each group $\Ext^i(S_\ga \boxtimes S_\b,X)$. Again, to deduce our Lemma, it will suffice to show that the boundary map
\[
 \Ext^1(S_\ga\boxtimes S_\b,S_\ga\boxtimes S_\b)\to \Ext^2(S_\ga\boxtimes S_\b,X)
\] is surjective.

The surjection $X\to S_\ga\boxtimes S_\b$ is known to induce an isomorphism
\[
 \Ext^2(S_\ga\boxtimes S_\b,X)\isomto \Ext^2(S_\ga\boxtimes S_\b,S_\ga\boxtimes S_\b).
\]
Hence it suffices to show that the composite map
\[
 \Ext^1(S_\ga\boxtimes S_\b,S_\ga\boxtimes S_\b)\to \Ext^2(S_\ga\boxtimes S_\b,S_\ga\boxtimes S_\b)
\]
is surjective. This composite map is given by multiplication by the class of (\ref{xext}) in $\Ext^1(S_\ga\boxtimes S_\b,X)$ followed by composition with the surjection $f:X\to S_\ga\boxtimes S_\b$. The resulting class in $\Ext^1(S_\ga\boxtimes S_\b,S_\ga\boxtimes S_\b)$ which is being multiplied by to produce our composite map is the class of the sequence
\[
 0\to S_\ga\boxtimes S_\b \to \Res_{\ga\b}S_\a/\mbox{ker}(f)\to S_\ga\boxtimes S_\b\to 0,
\] which is the class of 
the self-extension $Y$. Since this extension is non-split, the same argument as in the length two case establishes our desired surjectivity.
\end{proof}

\begin{lemma}\label{nonformal}
Suppose our KLR algebras are of type $B$, $C$, $F$ or $G$ and that we are working with the HMM-ordering.
Let $\a$ be a positive root that is not simple, and let $(\b,\ga)$ be the minimal pair for $\a$ given in the tables in the appendix.
Then the canonical map $\Ext^1(S_\ga\otimes S_\b,S_\a)\to \Ext^1(L_{\b\ga},S_\a)$ is nonzero. 
\end{lemma}
\begin{proof}
By inductive hypothesis applied to Proposition \ref{cuspidalselfext}, there exists a non-trivial extension $M_\ga$ of $S_\ga$ by $S_\ga$.
There are canonical injections $L_{\b\ga}\to S_\ga\otimes S_\b\to M_{\ga}\otimes S_\b$. Let $X$ be the quotient of $M_\ga\otimes S_\b$ by $L_{\b\ga}$. Using Lemma \ref{ses}, this is an extension of $S_{\ga}\otimes S_\b$ by $S_\a$.

We will find an element $x\in R(\a)$ and a word $\j\in \seq(\a)$ such that the action of $xe_{\ii_\a}$ on $X$ is a non-trivial map $e_{\ii_\a}X\to e_{\ii_\a}X$ factoring through a weight space $e_\j X$. The choice of word $\j$ will be such that $e_\j S_\a=0$. This makes it impossible for $L_{\b\ga}$ to be a submodule of $X$.

The image of the class of $X$ in $\Ext^1(L_{\b\ga},S_\a)$ is the kernel of the canonical map from $X$ to $S_\a$. If this image were zero, then $L_{\b\ga}$ would be a submodule of this kernel and hence a submodule of $X$, a contradiction.

There is a symmetry between $\b$ and $\ga$ in the above discussion, in that we could have equally well worked with a non-trivial self-extension of $S_\b$.

The table in the appendix lists the choice of $\b$ or $\ga$ for which the self-extension is used in this contruction, together with the desired element $x\in R(\a)$ and the weight $\j$. We give an en example of the computation required to show that the action of $xe_\ii$ has the appropriate properties.

Let us consider the second row in the F4 table. The only property we need to show which is not obvious is that the action of $xe_{\ii_\a}$ on $X$ is nontrivial. From the defining relations of the KLR algebra, we compute
\[
 xe_{\ii_\a}=\phi_1\phi_2\phi_3^2\phi_2\phi_1e_{1012}=(\phi_1\phi_2^2\phi_1y_1^2 - \phi_2\phi_1y_1\phi_2 - y_1 - \phi_1\phi_2\phi_1 y_3 )e_{1012}.
\]
Now we compute the action of this element on $M_\ga\otimes S_\b$. For this we need to know information about the structure of $M_\ga$ and $S_\b$, which is known by previous computations with smaller roots and the results of \cite{hmm}.

In this case, $y_1^2e_{1012}$, $\phi_2e_{1012}$ and $y_3e_{1012}$ each act by zero on $M_\ga\otimes S_\b$, since the corresponding elements act by zero in $M_\ga\boxtimes S_\b$, while $y_1 e_{1012}$ acts nontrivially because the same is true for its action on $M_\ga$. When passing to the quotient $X$ the action of $xe_{\ii_\a}$ is still nontrivial since $X$ and $M_\ga\otimes S_\b$ have the same $1012$-weight space.




There is one case not covered by the above argument, namely in $G_2$ with $\a=2\a_0+\a_1$, $\b=\a_0+\a_1$ and $\ga=\a_0$ where $\a_0$ is the short root and $\a_1$ is the long root.

In this case there is a five dimensional representation of $R(\a)$ with basis $v_{001}[3]$, $v_{001}[1]$, $v_{001}[-1]$, $v_{001}[-3]$, $v_{010}[0]$ where $v_\ii [d]$ is a vector of weight $\ii$ and degree $d$. The nonzero maps between these spaces are $\phi_1(v_{001}[3])=v_{001}[1]$, $\phi_1(v_{001}[-1])=v_{001}[-3]$, $y_1(v_{001}[i])=-v_{001}[i+2]$, $y_2(v_{001}[i])=v_{001}[i+2]$, $\phi_2(v_{001}[-3])=v_{010}[0]$ and $\phi_2(v_{010}[0])=v_{001}[3]$.

This representation represents an element of $\Ext^1(S_\ga\otimes S_\b,S_\a)$ with nonzero image in $\Ext^1(L_{\b\ga},S_\a)$.
\end{proof}

\begin{proposition}\label{cuspidalselfext}
 Let $\a$ be a positive root. Then  $\Ext^1(S_\a,S_\a)=k\{\a\cdot\a \}$ and
$\Ext^i(S_\a,S_\a)=0$ for $i\geq 2$.
\end{proposition}



\begin{proof}
Without loss of generality assume that our Cartan datum is irreducible.

When $\a$ a simple root, $R(\a)\cong k[x]$ with $x$ in degree $\a\cdot\a$ and so this Proposition is determinable by
a routine calculation. Henceforth we shall assume $\a$ is not simple.

First let us suppose our Cartan datum is of type $B$, $C$, $F$ or $G$ and that we are using the HMM ordering. Let $(\b,\ga)$ be the minimal pair for $\a$ as tabulated in the appendix.

First consider the long exact sequence obtained by applying $\Hom(-,S_\a)$ to the sequence (\ref{ses2}). By Lemma \ref{ressa},  $\Res_{\beta,\gamma}S_\a=0$.
By adjunction this implies $\Ext^i(S_\b \otimes S_\ga , S_\a)=0$ for all $i\geq 0$. Thus this long exact sequence reduces to a sequence of isomorphisms 
\begin{equation}\label{dimshift}
\Ext^i(S_\a,S_\a)\isomto \Ext^{i+1}(L_{\b\ga},S_\a)
\end{equation}
for each $i\geq 0$.

In particular $\Ext^i(L_{\b\ga},S_\a)$ is one dimensional, so Lemma \ref{nonformal} implies that the map $\Ext^1(S_\ga\otimes S_\b,S_\a)\to \Ext^1(L_{\b\ga},S_\a)$ is surjective.

Now consider the long exact sequence obtained by applying $\Hom(-,S_\a)$ to the sequence (\ref{ses1}). By the above surjectivity result, and the fact that $L_{\b\ga}$ and $S_\a$ are nonisomorphic simple representations, we obtain the short exact sequence
\[
 0\to \Ext^1(S_\a,S_\a)\to \Ext^1(S_\ga\otimes S_\b,S_\a)\to\Ext^1(L_{\b\ga},S_\a)\to 0.
\]

By Lemma \ref{minpairext} the middle term is two-dimensional and we have just shown that the quotient is one-dimensional.
This proves the $\Ext^1$ part of this Proposition up to a grading shift.

We now return to the isomorphisms (\ref{dimshift}). These isomorphisms are given by left multiplication by the class of the sequence (\ref{ses2}).
Consider the product map
\[
 \Ext^1(L_{\b\ga},S_\a)\times \Ext^1(S_\a,S_\a)\to \Ext^2(L_{\b\ga},S_\a).
\]
As we've just shown $\dim\Ext^1(S_\a,S_\a)=1$, again using (\ref{dimshift}), we see $\dim\Ext^2(L_{\b\ga},S_\a)=1$ and that this product map is surjective.

There is a commutative diagram as follows, where the vertical maps are given by the Yoneda product, and the horizontal maps $\delta$ are the connecting maps in the long exact sequence obtained by applying $\Hom(-,S_\a)$ to the sequence (\ref{ses1}):
\[\begin{CD}
 \Ext^1(L_{\b\ga},S_\a)\otimes  \Ext^1(S_\a,S_\a)@>\delta\otimes \operatorname{id}>> \Ext^2(S_\a,S_\a)\otimes\Ext^1(S_\a,S_\a)\\
@VVV    @VVV \\
\Ext^2(L_{\b\ga},S_\a) @>\delta>> \Ext^3(S_\a,S_\a).
\end{CD}\]
In this diagram, we know that the uppermost boundary map is zero and the leftmost product map is surjective. Therefore the lower boundary map is also zero.

The vanishing of $\Ext^i(S_\a,S_\a)$ for $i\geq 2$ can now be deduced in a routine fashion using Lemma \ref{minpairext} and the two long exact sequences we have considered.

Apart from the grading on $\Ext^1(S_\a,S_\a)$ which we shall postpone until the end of this section, this completes the proof in this particular case.
So now we shift our focus to the general case.

We may assume that the KLR algebras have finite global dimension. In types $B$, $C$, $F$ or $G$, this is a consequence of Theorem \ref{finitecd} proved using the HMM ordering. In types $A$, $D$ and $E$ over a ground field of characteristic zero, this is due to \cite[Theorem A]{kato}. An alternative approach in this latter case using the techniques of this section is proved to be possible in \cite{bkm}.

Let $(\b,\ga)$ be a minimal pair for $\a$. By the proof of Theorem \ref{main}(1), there are short exact sequences
\begin{eqnarray}
\label{qext} 0 \to L_{\b\ga} \to S_\ga\otimes S_\b \to Q \to 0 \\
\label{kext} 0 \to K \to S_\b \otimes S_\ga \to L_{\b\ga} \to 0
\end{eqnarray}
where $K$ and $Q$ have only $S_\a$ appearing as a simple factor in any composition series.

By induction on the height of the root, $\Ext^i(S_\ga\boxtimes S_\b,S_\ga\boxtimes S_\b)=0$ for $i\geq 3$. Lemma \ref{resgammabeta} implies that $\Ext^i(S_\ga\boxtimes S_\b,\Res_{\ga\b}S_\a)=0$ for $i\geq 3$ and hence by adjunction $\Ext^i(S_\ga\otimes S_\b,S_\a)=0$ for $i\geq 3$ 

As in the specialised case considered above, $\Ext^i(S_\b\otimes S_\ga,S_\a)=0$ for all $i$. 

We now apply $\Hom(-,S_\a)$ to the two short exact sequences (\ref{qext}) and (\ref{kext}) to obtain two long exact sequences. The vanishing results we have just 
deduced imply that there are isomorphisms 
\begin{eqnarray*}
\Ext^i(L_{\b\ga},S_\a)\cong \Ext^{i+1}(Q,S_\a)& \mbox{   for}& i\geq 3, \\
 \Ext^i(K,S_\a)\cong \Ext^{i+1}(L_{\b\ga},S_\a)&\mbox{for all}&i\geq 0.
\end{eqnarray*}

Let $d$ be the maximal integer with $\Ext^d(S_\a,S_\a)\neq 0$, which exists since $R(\a)$ has finite global dimension. 

As $K$ is isotypic, $\Ext^d(K,S_\a)\neq 0$, implying $\Ext^{d+1}(L_{\b\ga},S_\a)\neq 0$. If $d\geq 2$, we obtain $\Ext^{d+2}(Q,S_\a)\neq 0$, and since $Q$ is also isotypic, $\Ext^{d+2}(S_\a,S_\a)\neq 0$, contradicting the maximality of $d$.

Thus we have shown that $\Ext^i(S_\a,S_\a)=0$ for $i\geq 2$. We shall defer the rest of the proof of this Proposition until the end of this section.
\end{proof}

We now state our main theorem of this section.
\begin{theorem}\label{finitecd}
As a graded algebra, the KLR algebra $R(\nu)$ has global dimension $|\nu|$.
\end{theorem}

This will be dealt with by a sequence of Lemmas. 

\begin{lemma}
Suppose $m,m'\in \N^{\Phi^+}$ with $|m|=|m'|=\nu$. Then $\Ext^i(\nabla(m),\Delta(m'))=0$ for $i>|\nu|$. 
\end{lemma}

\begin{proof}
If either $\nabla(m)$ or $\Delta(m')$ is not cuspidal, then this follows from the adjunction of Corollary \ref{extadjunction}, together with applying an induction on $\nu$ to Theorem \ref{finitecd}.
If both $\nabla(m)$ and $\Delta(m')$ are cuspidal, they must be isomorphic up to a grading shift, and this Lemma is implied by the more presice Proposition \ref{cuspidalselfext}.
\end{proof}

\begin{lemma}
 Suppose $m,m'\in \N^{\Phi^+}$ with $|m|=|m'|=\nu$. Then $\Ext^i(L(m),L(m'))=0$ for $i>|\nu|$. 
\end{lemma}

\begin{proof}
 First we prove by induction on $m'$ that $\Ext^i(\nabla(m),L(m'))=0$ if $i>|\nu|$.

We apply $\Hom(\nabla(m),-)$ to the short exact sequence $0\to K\to \Delta(m')\to L(m')\to 0$ to obtain a long exact sequence of Ext-groups, with $\Ext^i(\nabla(m),L(m'))$ sandwiched between $\Ext^i(\nabla(m),\Delta(m'))$ and $\Ext^{i+1}(\nabla(m),K)$. The first of these two groups is zero by the previous Lemma. By the inductive hypothesis $\Ext^{i+1}(\nabla(m),L(m''))=0$ for $m''<m'$. By Theorem \ref{main}(5), the composition factors of $K$ are all of the form $L(m'')$ for $m''<m$, hence $\Ext^{i+1}(\nabla(m),K)=0$, as required.

We proceed from this intermediate result to the statement of the lemma by a similar inductive argument utilising the short exact sequence $0\to L(m)\to \nabla(m)\to Q\to 0$.
\end{proof}

\begin{corollary}\label{fdim}
 If $A$ and $B$ are finite dimensional representations of $R(\nu)$, then $\Ext^i(A,B)=0$ for $i>|\nu|$.
\end{corollary}
\begin{proof}
 This is immediate using the classification result Theorem \ref{main}(3).
\end{proof}

\begin{lemma} \label{fgen}
 Let $A$ and $B$ be finitely generated representations of $R(\nu)$. Then $\Ext^i(A,B)=0$ for $i>|\nu|$.
\end{lemma}

\begin{proof} 
First let us consider the case where $B$ is finite dimensional.
It is sufficient to show that the degree zero piece of $\Ext^i(A,B)$ is zero for $i>|\nu|$.
 Let $A'$ be the submodule of $A$ generated by $\oplus_{d>c}A_d$ where $c$ is a sufficiently large integer. Then $A/A'$ is finite dimensional, so by Corollary \ref{fdim}, if $\Ext^i(A,B)\neq 0$, then $\Ext^i(A',B)\neq 0$. Consider a minimal projective resolution of $A'$. Since $R(\nu)$ is almost positively graded, the $i$-th term of this resolution is concentrated in degrees above those appearing in $B$ for sufficiently large $c$. Thus for $c$ sufficiently large, there are no degree zero homomorphisms to $B$ from the $i$-th term of this projective resolution, so $\Ext^i(A',B)$ vanishes in degree zero.

A similar argument allows us to relax the condition that $B$ is finite dimensional to $B$ is finitely generated.
\end{proof}

We now give the proof of Theorem \ref{finitecd}.

\begin{proof}
By \cite[Theorem 4.1.2]{weibel}, it suffices to show that if $M$ is a finitely generated $R(\nu)$-module and $C$ is an arbitrary $R(\nu)$-module, then $\Ext^i(M,C)=0$ if $i>|\nu|$. It is also sufficient to consider the degree zero piece of $\Ext^i(M,C)$.

The algebra $R(\nu)$ is Noetherian \cite[Corollary 2.11]{khovanovlauda}. Therefore we can find a projective resolution of $M$ consisting of finitely generated projective modules. Consider such a resolution $P_{\bullet}\to M$. Any element of $\Ext^i(M,C)$ is represented by a map from $P_i$ to $C$. Since $P_i$ is finitely generated, this map factors through a finitely generated submodule $C'$ of $C$.
Thus if $\Ext^i(M,C)\neq 0$, then $\Ext^i(M,C')\neq 0$ for some finitely generated submodule $C'$ of $C$. So without loss of generality, we may assume that $C$ is finitely generated.



By Lemma \ref{fgen}, $\Ext^i(M,C)=0$ for $i> |\nu|$.

If $X$ is a tensor power of cuspidals associated to simple roots, then $\Ext^{|\nu|}(X,X)\neq 0$, so the upper bound we have proven on the global dimension is the best possible.
\end{proof}





\subsection{The Ext Bilinear Form}
Given two representations $X$ and $Y$ of $R(\nu)$, finitude of global dimension allows us to define
\[
 \langle X, Y \rangle = \sum_{i=0}^\infty (-1)^i\dimq \Ext^i(X,Y).
\]
This descends to a pairing on the Grothendieck group
\[
 \langle \cdot,\cdot\rangle \map{\f_{\A}^*\times \f_{\A}^*}{\A}.
\]
This form is $\A$-semilinear in the sense that for any $f,g\in\A$ and $u,v\in \f_{\A}^*$, $$\langle fv,gw\rangle=f\bar g\langle v,w\rangle.$$

It is straightforward to compute
\[
 \langle [S_{\a_i}],[S_{\a_i}]\rangle = 1-q_i^2
\] for any $i\in I$.
Furthermore, the induction-restriction adjunction implies that
\[
 \langle x x',y\rangle = \langle x \otimes x',r(y) \rangle,
 \]
 for all $x,x',y\in \f^*_\A$.

Working now over $\Q(q)$ and using the isomorphism between $\f$ and $\f^*$ induced by the bilinear form $(\cdot,\cdot)$, we may define a new form $(\cdot,\cdot)'\map{\f\times \f}{\Q(q)}$ by
\[
 (x,y)'=\langle x,\bar y \rangle.
\]
Here, we caution the reader that the bar involution in the above equation must be taken to be the bar involution on $\f^*$ (as opposed to the involution on $\f$).

Since $[S_{\a_i}]$ is bar-invariant and equal to $(1-q_i^2)E_i$, the above remarks show that $(\cdot,\cdot)$ is $Q(q)$-bilinear and satisfies
\begin{eqnarray*}
 (E_i,E_i)'&=&(1-q_i^2)^{-1} \\
(xx',y)'&=&(x\otimes x',r(y))'.
\end{eqnarray*}
Since $(\cdot,\cdot)$ is the unique bilinear form on $\f$ with these properties, we have $(\cdot,\cdot)'=(\cdot,\cdot)$.

We can now use the facts $(E_\a,E_\a)=(1-q_\a^2)^{-1}$ and $E_\a^*=(1-q_\a^2)E_\a$ mentioned in Theorem \ref{pbwbasis} to deduce
\[
 \langle E_\a^*,E^*_\a \rangle  = 1-q_\a^2.
\]

Since $[S_\a]=E_\a^*$, we can use this last computation to finish the proof of Proposition \ref{cuspidalselfext}.
We've already shown that $\Ext^i(S_\a,S_\a)=0$ for $i\geq 2$ and as $S_\a$ is simple, $\Ext^0(S_\a,S_\a)=k$. Thus the computation of this pairing shows that $\Ext^1(S_\a,S_\a)=k\{\a\cdot\a\}$ as required.

\section{Appendix}

This appendix contains data needed for Lemmas \ref{ses}, \ref{minpairext} and \ref{nonformal}.

We follow the notation of \cite{hmm} and label the vertices of our Dynkin diagram with the integers $0,1,\ldots,r-1$.

In the tables below, in lieu of writing the root $\a$, we give the corresponding word $\ii_\a$. The first column comprises a positive root $\a$ that is not simple. The second and third columns give a minimal pair $(\b,\ga)$ for $\a$. The remaining columns contain data used in the proof of Lemma \ref{nonformal}.

These tables should be consulted in conjunction with those of \cite[\S 5]{hmm} which contain the characters of the cuspidal representations for the HMM-ordering. In each case, the computation should proceed by induction on the height of the root $\alpha$, since some structure of the self extensions of smaller cuspidal modules is needed as input.

First, for type $B_r$, we label our Dynkin diagram as follows:
 \[
  \includegraphics[scale=1.1]{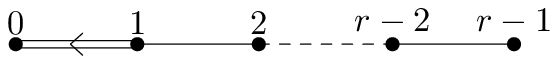}
 \]
\[
 \begin{tabular}{c c|c|c|c|c|c}
$\a$ & & $\ga$ & $\b$ & & $x$ & $\uj$ \\ \hline
  $i\ldots j$ & $0\leq i< j <r$ & $i\ldots j-1$ & $j$ & $\beta$ & $\phi_{j-i}^2$ & $i\ldots j-2,j,j-1$ \\
$00\ldots k$ & & 0 & $0\cdots k$ & $\ga$ & $\phi_1\phi_2^2$ & $0102\ldots k$ \\
$j\cdots 00\cdots k$ & $0<j<k$ & $j$ & $j-1\cdots 00\cdots k$ & $\ga$ & $\phi_1^2$ & $j-1,j,j-2\cdots 00\cdots k$
 \end{tabular}
\]

For type $C_r$, we label our Dynkin diagram as follows:
 \[
 \includegraphics[scale=1.1]{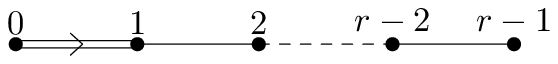}
\]
In the following table of data, there is the restriction that $j\geq 2$ in both the second and fourth rows.
\[
 \begin{tabular}{c c|c|c|c|c|c}
$\a$ & & $\ga$ & $\b$ & & $x$ & $\uj$ \\ \hline
 01 & & 0 & 1 & $\ga$ & $\phi_1^2$ & 10 \\
$i\cdots j$ &  & $i\cdots j-1$ & $j$ & $\b$ &$\phi_{j-i}^2$ & $i\cdots j-2,j,j-1$ \\
$101\cdots k$ &  & 1 & $01\cdots k$  & $\ga$ &$\phi_1\phi_2\phi_1$ & $011\cdots k$ \\
$j\cdots 101\cdots k$ &  & $j$ &$j-1\cdots 0\cdots k$ & $\ga$ & $\phi_1^2$ & $j-1,j,j-2\cdots 0\cdots k$ \\
$0\cdots j1\cdots j$ &  &$0\cdots j$&$1\cdots j$&$\b$&$\phi_{j+1}\cdots \phi_2\phi_1^2\phi_2\cdots \phi_{j+1}$& $10\cdots j2\cdots j$\\ 
 \end{tabular}
\]

For type $F_4$, we label our Dynkin diagram as follows:
 \[
  \includegraphics[scale=1.1]{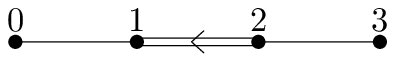}
 \]

Note that the roots labelled 1,2 and 3 form a Dynkin diagram of type $B_3$ with the HMM ordering. So we only need to deal with roots $\a$ for which the letter 0 appears in $\ii_\a$.

\[
\begin{tabular}{  c|c| c|c|c| c }
  $\a$ & $\ga$ & $\b$ & & $x$ & $\uj$ \\ \hline
0123 & 0 & 123  &$\ga$&$\phi_1^2$ & 1023\\  
  1012 & 1 & 012 &$\ga$& $\phi_1\phi_2\phi_3^2\phi_2\phi_1$ & 0121\\
01012 &  01 & 012 &$\ga$& $\phi_2\phi_3\phi_4^2\phi_3\phi_2$ & 00121\\
10123 & 1 & 0123 &$\ga$& $\phi_1\phi_2\phi_3^2\phi_2\phi_1$ & 01213 \\
010123 & 01 & 0123 &$\ga$& $\phi_2\phi_3\phi_4^2\phi_3\phi_2$ & 001213\\
210123 & 2& 10123 & $\ga$& $\phi_1^2$ & 120123 \\
1210123 & 1 & 210123 & $\ga$&$\phi_1\phi_2\phi_3^2\phi_2\phi_1$ & 2101123 \\
2010123 & 2  & 010123 &$\ga$& $\phi_1\phi_2^2\phi_1$ & 0120123\\
12010123 & 12 & 010123 & $\b$&$\phi_4\phi_3\phi_5\phi_4\phi_2\phi_3^2\phi_2$ & 10120123 \\
112010123 & 1 & 12010123 &$\ga$& $\phi_1\phi_2^2$ & 121010123 \\
2112010123 & 2 & 112010123 & $\ga$& $\phi_1^2$ & 1212010123\\
21012310123 & 210123 & 10123 & $\ga$&$\phi_6\phi_5\phi_4\phi_3\phi_2\phi_1^2 \phi_2\phi_3\phi_4\phi_5\phi_6$ &12101230123 \\
\end{tabular}
\]

%
%
%
%
%
%
%
%
%
%
%

For type $G_2$, label the short root with 0 and the long root with 1.
\[
\begin{tabular}{  c|c| c|c|c| c }
  $\a$ & $\ga$ & $\b$ & & $x$ & $\uj$ \\ \hline
  01 & 0 & 1 & $\b$ & $\phi_1^2$& 10 \\
  001 & 0 & 01 & $\ga$ & - & - \\
0001 & 0 & 001 & $\ga$ & $\phi_3^2\phi_2\phi_1$ & 0010  \\
00101 &  001 & 01 & $\ga$ & $\phi_2^2\phi_1\phi_3\phi_4\phi_2\phi_3$ & 01001\\
\end{tabular}
\]


\begin{thebibliography}{Rou12}

\bibitem[B]{bourbaki}
N.~Bourbaki.
\newblock {\em \'{E}l\'ements de math\'ematique. {F}asc. {XXXIV}. {G}roupes et
  alg\`ebres de {L}ie. {C}hapitres {IV}--{VI}}.
\newblock Actualit\'es Scientifiques et Industrielles, No. 1337. Hermann,
  Paris, 1968.

\bibitem[BKM]{bkm}
J. Brundan, A. Kleshchev and P. J. McNamara.
\newblock Homological properties of finite type Khovanov-Lauda-Rouquier algebras.
\newblock \arxiv{1210.6900}

\bibitem[HMM]{hmm}
David Hill, George Melvin, and Damien Mondragon.
\newblock Representations of quiver {H}ecke algebras via {L}yndon bases.
\newblock {\em J. Pure Appl. Algebra}, 216(5):1052--1079, 2012.
\newblock \arXiv{0912.2067}.

\bibitem[K]{kato}
Syu Kato.
\newblock {PBW} bases and {KLR} algebras.
\newblock \arXiv{1203.5245}.

\bibitem[KL1]{khovanovlauda}
Mikhail Khovanov and Aaron~D. Lauda.
\newblock A diagrammatic approach to categorification of quantum groups. {I}.
\newblock {\em Represent. Theory}, 13:309--347, 2009.
\newblock \arxiv{0804.2080}.


\bibitem[KL2]{klII}
Mikhail Khovanov and Aaron~D. Lauda.
\newblock A diagrammatic approach to categorification of quantum groups {II}.
\newblock {\em Trans. Amer. Math. Soc.}, 363(5):2685--2700, 2011.
\newblock \arxiv{0807.3250}.

\bibitem[KR]{kleshchevram}
Alexander Kleshchev and Arun Ram.
\newblock Representations of {K}hovanov-{L}auda-{R}ouquier algebras and
  combinatorics of {L}yndon words.
\newblock {\em Math. Ann.}, 349(4):943--975, 2011.
\newblock \arxiv{0909.1984}.

\bibitem[LV]{laudavazirani}
Aaron~D. Lauda and Monica Vazirani.
\newblock Crystals from categorified quantum groups.
\newblock {\em Adv. Math.}, 228(2):803--861, 2011.
\newblock \arxiv{0909.1810}.

\bibitem[LS]{ls}
Serge Levendorski{\u\i} and Yan Soibelman.
\newblock Algebras of functions on compact quantum groups, {S}chubert cells and
  quantum tori.
\newblock {\em Comm. Math. Phys.}, 139(1):141--170, 1991.

\bibitem[L]{lusztigbook}
George Lusztig.
\newblock {\em Introduction to quantum groups}, volume 110 of {\em Progress in
  Mathematics}.
\newblock Birkh\"auser Boston Inc., Boston, MA, 1993.

\bibitem[R1]{rouquier}
R.~Rouquier.
\newblock 2-{K}ac-{M}oody algebras.
\newblock \arxiv{0812.5023}.

\bibitem[R2]{rouquier2}
Rapha\"el Rouquier.
\newblock {Quiver Hecke algebras and 2-Lie algebras.}
\newblock {\em Algebra Colloq.}, 19(2):359--410, 2012.
\newblock \arxiv{1112.3619}.

\bibitem[TW]{tingleywebster}
Peter Tingley and Ben Webster.
\newblock Mirkovic-{V}ilonen polytopes and {K}hovanov-{L}auda-{R}ouquier
  algebras.
\newblock \arxiv{1210.6921}.

\bibitem[VV]{vv}
M.~Varagnolo and E.~Vasserot.
\newblock Canonical bases and {KLR}-algebras.
\newblock {\em J. Reine Angew. Math.}, 659:67--100, 2011.
\newblock \arxiv{0901.3992}.

\bibitem[W]{weibel}
Charles~A. Weibel.
\newblock {\em An introduction to homological algebra}, volume~38 of {\em
  Cambridge Studies in Advanced Mathematics}.
\newblock Cambridge University Press, Cambridge, 1994.

\end{thebibliography}

\end{document}